\documentclass[12pt]{amsart}
\usepackage{latexsym,fancyhdr,amssymb,color,amsmath,amsthm,graphicx,listings,comment}
\usepackage[section]{placeins}
\newtheorem{thm}{Theorem}
\newtheorem{lemma}{Lemma}

\let\paragraph\subsection

\title{Characteristic Topological Invariants} 
\author{Oliver Knill}
\date{February 5, 2023}
\address{Department of Mathematics \\ Harvard University \\ Cambridge, MA, 02138 }
\subjclass{}

\keywords{Topological invariant, Green function, Higher characteristics }

\begin{document}
\maketitle

\begin{abstract}
The higher characteristics $w_m(G)$ for a finite abstract simplicial complex $G$
are topological invariants that satisfy $k$-point Green function identities and
can be computed in terms of Euler characteristic in the case of closed manifolds,
where we give a new proof of $w_m(G)=w_1(G)$. Also the sphere formula generalizes:
for any simplicial complex, 
the total higher characteristics of unit spheres at even dimensional simplices is
equal to the total higher characteristic of unit spheres at odd dimensional simplices.
\end{abstract}

\section{Summary}

\paragraph{}
We prove {\bf $k$-point Green function formulas} 
$$ w_m(G)=\sum_{x_j \in G} g_m(x_1,\dots,x_k) 
        = \sum_{x \in G^k} \prod_{j=1}^k w(x_j) w_m(\bigcap_{j=1}^k U(x_j)) $$
for the {\bf $m$'th characteristic} of $A \subset G$
$$ w_m(A) = \sum_{x \in A^m, \bigcap_j x_j \in A} \prod_{j=1}^m w(x_j)  \; , $$
where $w(x)=(-1)^{{\rm dim}(x)}$ and where $G$ is a finite abstract simplicial complex.
The {\bf stars} $U(x)=\{ y, x \subset y\}$ define a topological base of a finite non-Hausdorff
{\bf topology} on $G$. For two arbitrary open sets $U,V$, the {\bf valuation formula} 
$$  w_m(U \cup V) + w_m(U \cap V) = w_m(U) + w_m(V) $$
holds for all $m \geq 1$. This makes the higher characteristics $w_m$ 
{\bf topological invariants}: they agree for homeomorphic complexes. It also makes more
general energized versions of $w_m$ sheaf ready. For {\bf manifolds}, $w_m(M)=w_1(M)$. 

\begin{figure}[!htpb]
\scalebox{0.5}{\includegraphics{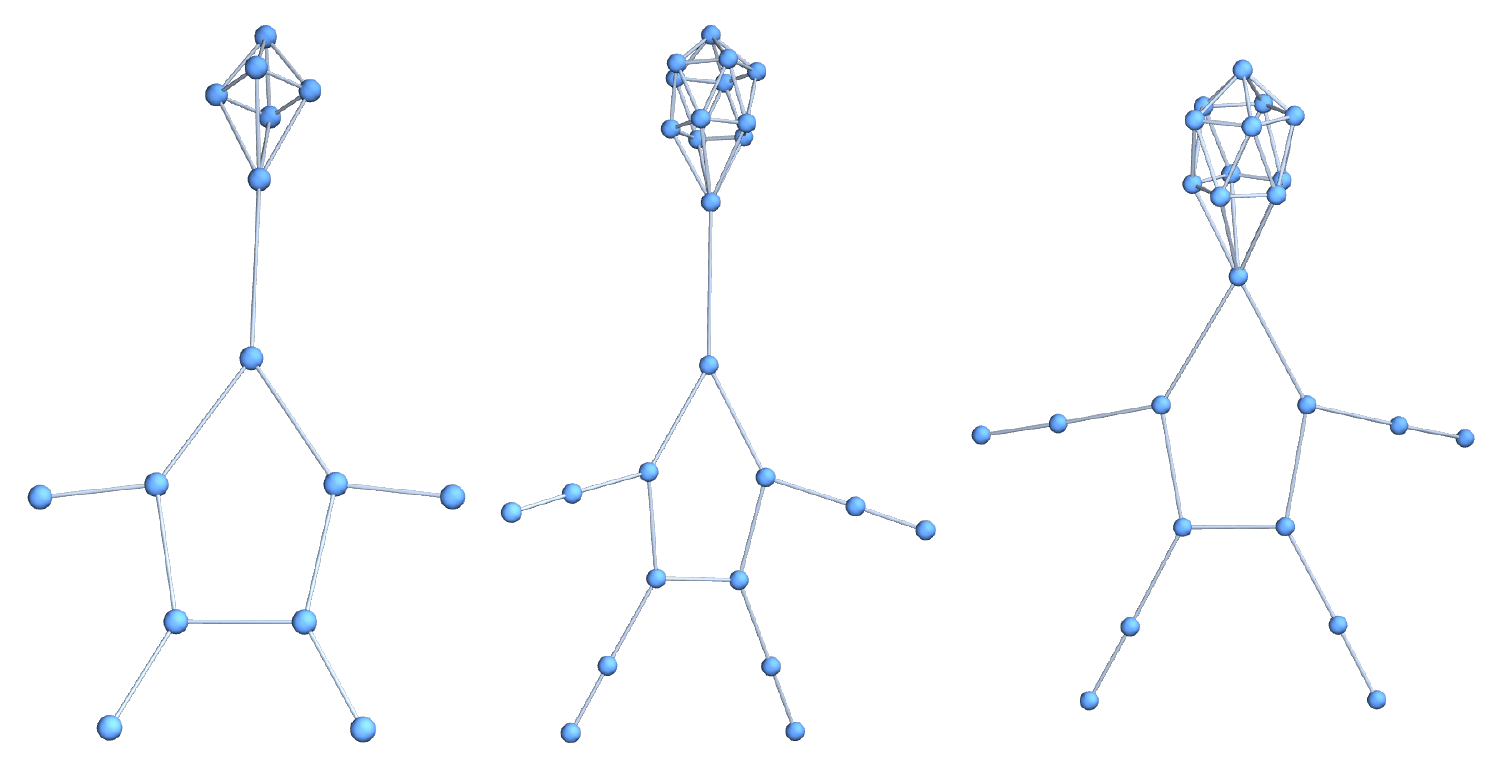}}
\caption{
These complexes $G,H,K$ are all homotopic to a {\bf wedge sum} of a 2-sphere and a 1-sphere. 
The first two complexes $G,H$ are homeomorphic. The complex $K$ however is not homeomorphic to $G$.
We have Euler $\omega_1(G)=\omega_1(H)=\omega_1(K)=1$ and Betti vector $b_1(G)=b_1(H)=b_1(K)=(1,1,1)$.
For Wu $\omega_2(G)=\omega_2(H)=11, b_2(G)=b_2(H)=(0,0,11,1,1)$ and $\omega_2(K)=9, b_2(K)=(0,0,11,3,1)$.
Then $\omega_3(G)=\omega_3(H)=-29, b_3(G)=b_3(H)=(0,0,0,31,3,2,1)$ and
$\omega_3(K)=-23, b_3(K)=(0,0,0,31,12,5,1)$. 
}
\end{figure}

\paragraph{}
We use short hand $w_m(A)=\sum_{X \in A^m, \bigcap X \in A} w(X)$, summing over 
$X=(x_1,\dots,x_m) \in A^m$ with $\bigcap X \in A$, meaning $\bigcap_{j=1}^m x_j \in A$ and 
$w(X)=\prod_{j=1}^m (-1)^{{\rm dim}(x_j)}$. 
The {\bf Euler characteristic} is $w_1(A)$ and the {\bf Wu characteristic} is $w_2(A)$. 
{\bf $k$-point Green function} maps a k-point configuration $X \in G^k$ to 
$g_m(X)=w(X) w_m(U(X))$ with $U(X)=\bigcap_{j=1}^k U(x_j))$.
The {\bf energy formula} $w_m(G) = \sum_{X \in G^k} g_m(X)$
involves open sets in $G$ and so is sheaf theoretic. Unlike in $w_m(G)$, 
there is no $\bigcap X \in G$ assumption in the Green part. 
The closure $B(x)$ of a {\bf star} $U(x)$ is called the {\bf unit ball}. Its boundary 
$S(x)$ is the {\bf unit sphere} also known as {\bf link}. 
The {\bf local valuation formula} 
$$ w_m(B(x)) = w_m(U(x)) - (-1)^m w_m(S(x)) \;  $$
holds for all stars $U(x)$, but in general not for arbitrary open sets $U$ with 
with $B(X)=\overline{U(X)}$ and $S(X)=B(X)\setminus U(X)=\delta U(X)$. 

\paragraph{}
The special {\bf one-particle case} $k=1$ gives for all $m \geq 1$
a Gauss-Bonnet or Poincar\'e-Hopf type formula 
$$  w_m(G)=\sum_{x \in G} w(x) w_m(U(x)) \; , $$ 
which reduces the computation of $w_m(G)$ to local expressions.
We can think of $g_m(x)$ as a ``curvature" or interpret it as an ``index". Note that the Gauss-Bonnet
type formula for open sets does not generalize if $G$ is replaced by an open set $U$. It only
works if $G$ is replaced by a smaller simplicial complex. The unit balls $B(x)=\overline{U(x)}$ 
for example are by themselves again simplicial complexes and we also have
$w_m(G) = \sum_{x \in G} w(x) w_m(B(x))$. 

\paragraph{}
For all complexes $G$, the {\bf sphere formula} 
$$  \sum_{x \in G} w(x) w_m(S(x))=0  $$
holds. We have seen that in \cite{KnillEnergy2020,Sphereformula} for $m=1$.
It follows from the local valuation formula and the {\bf ball formula}
$w_m(G) = \sum_{x \in G} w(x) w_m(B(x))$. This is can be useful for large complexes,
where we can use the ball formula again for computing $w_m(B(x))$, especially if many $B(x)$ have
the same type or are known. The sphere formula for example proves that if $G$ is a complex for 
which the  characteristic $w_m$ of all spheres $S(x)$ is constant $c \neq 0$ then $w_m(G)=0$. This applies
for odd-dimensional manifolds or more generally for odd-dimensional {\bf Dehn-Sommerville d-manifolds}
$\mathcal{X}_d$ \cite{DehnSommerville} recursively defined by the property
 $\chi(G)=1+(-1)^d$ and that all unit spheres satisfy $S(x) \in \mathcal{X}_{d-1}$. 
The induction starts with $\mathcal{X}_{-1}=\{0\}$, where the {\bf void} $0=\{\}$ is the empty complex, 
which is also the $(-1)$-sphere. 

\paragraph{}
Unlike manifolds, Dehn-Sommerville
spaces $\mathcal{X}=\bigcup \mathcal{X}$ define a {\bf monoid} under the {\bf join operation}
$G+H=G \cup H \cup \{ xy, x \in G, y \in H \}$. It contains the 
{\bf sphere monoid} $\mathcal{S}$ of all spheres in which the $(-1)$ sphere $0$ is the zero 
element. The join of a $p$-spheres with a $q$-sphere is a $(p+q+1)$-sphere 
using induction and $S_{G+H}(x)=S_G(x) + H$ for $x \in G$ or $S_{G+H}(x)=G+H(x)$ for $x \in G$.  
It is natural to look at generalized $d$-manifolds which are complexes for which all 
unit spheres are in $\mathcal{X}_{d-1}$. This is a larger class than $d$-manifolds
but which like manifolds is no sub-monoid of the monoid $(\mathcal{G},+)$ of all complexes.

\paragraph{}
We have proven the energy formula in the case $m=1,k=2$ already 
as $\chi(G) = \sum_{x,y} g(x,y)$, where 
$$ g(x,y) = w(x) w(y) \chi(U(x) \cap U(y)) $$ 
are the matrix entries of the inverse of the {\bf connection matrix} $L(x,y)=1$ if 
$x \cap y \neq \emptyset$ and $L(x,y)=0$ else. 
In the case $m=1,k=1$, we have seen the super trace expression 
$\chi(G)=w_1(G)=\sum_{x \in G} w(x) \chi(U(x))$ involving $g(x,x) = \chi(U(x))$. 
We then generalized the energy formula to energized complexes, where $w(x)$ was
replaced by an arbitrary ring-valued function $h(x)$. We had explored such 
{\bf energized complexes} also in the case when $h$ was division algebra valued. 

\paragraph{}
The energy theorem can more generally be seen as a linear 
transformation which turns an {\bf interaction energy function}
$h_m: G^m \to R$ on intersecting points to a
Green function $g_m: G^k \to R$ on k-tuples of {\bf open stars}. The parameters $m$ and $k$
can be arbitrary positive integers. The energy 
$$ w_m(A) = \sum_{X \in A^m, \bigcap X \in A} h_m(X) $$ 
then defines {\bf $k$-point potential values} 
$$ g_m(X) = w(X) w_m(U(X)) $$ 
or $X \in G^k$. We then still have for any $m,k \geq 1$ the relation
$$ w_m(G) = \sum_{X \in G^m, \bigcap X>0} h_m(X) = \sum_{X \in G^k} g_m(X) \; . $$ 
Of course, for a general choice of the interaction energy $h$, there is no topological invariance.
Indeed, already requiring the weaker combinatorial property of being invariant under
Barycentric refinements forces $h$ to be of the form $h(x)=w(x)$. 

\paragraph{}
The more general picture with {\bf energized local interaction} $h_m$ explains, 
why a positive semi-definite $h$ produces positive semi-definite $g$: if $h_m$ is
zero except for some fixed $X_0$, where $h_m(X_0)=1$, then $g_m: X \in G^k \to g_m(X)$ 
is positive semi definite form. Convex combinations of positive semi definite forms $h_m$ 
in $m$ variables are mapped into convex combinations of positive semi definite forms $g_m$ 
in $k$ variables. In the case $m=1,k=2$ for example, where 
the Green function matrix $g(x,y)$ is the inverse of the connection matrix $L(x,y)$, we have
a duality between positive definite quadratic forms. In this case, $L$ 
and $g$ are even isospectral. They gave isospectral positive definite
integral quadratic forms. One can write $L(x,y) = w_m(V(x) \cap V(y))$ with $V(x)=\overline{ \{x\} }$. 
For $|h(x)|=1,k=2,m=1$ we had them inverses of each other. 
\cite{CountingMatrix,GreenFunctionsEnergized,EnergizedSimplicialComplexes}.

\paragraph{}
As for references, we should add that the subject has a large cultural background starting with 
the Max Dehn and Poul Heegaard \cite{DehnHeegaard} who first defined abstract simplicial complexes in 1907. 
Together with finite topological spaces as first considered by Pavel Alexandrov \cite{Alexandroff1937}
in 1937, this produces a powerful framework. We pointed out some reasons
in \cite{Sphereformula,FiniteTopology} why we do not want to look at geometric realizations. An example
is that the geometric realization of the double suspension of a Poincar\'e homology sphere is 
homeomorphic to a 5-sphere even so the combinatorial complexes are obviously not homeomorphic in 
a discrete sense. To remain in the finite topology allows to avoid difficulties from the continuum.

\paragraph{}
The mathematics of  Dehn,Heegard or Alexandrov  is not only 
mathematically simpler as it is only combinatorics of finite sets of sets only, we can 
also branch off into area, where the continuum is awkward. 
For example, we can look at the {\bf cohomology of open sets} in a simplicial complex. This is 
richer than the cohomology of closed sets. Since the Betti vector of a single set 
$\{x\}$ (which by definition is an open set $\{ x \}$ with locally maximal $x$ 
in a simplicial complex) is the basis vectors $e_{|x|}$, where $|x|$ is the 
cardinality of the set, we can realize any Betti vector with open sets. We would not know how
to realize a given Betti vector with a simplicial complex. 
In the continuum, the cohomology of an open set is usually only considered in the 
sense of a limit of cohomologies of compact subsets. Even the homotopy is different. 
An open interval $\{ \{1,2\} \}$ for example has Euler characteristic $-1$
and is not homotopic to a single point $\{ \{1\} \}$ which has Euler characteristic $1$. 
These two spaces should not be considered homotopic, for any sensible definition of homotopy. 
But we can identify $A=\{ \{1,2\} \}$ for example with $B=\{ \{2\},\{1,2\},\{2,3\} \}$ which 
is also an open interval. The open sets $A,B$ are homeomorphic. 
Their boundaries $\delta A = \overline{A} \setminus A = \{ \{1\},\{2\} \}$  
and $\delta B = \overline{A} \setminus B = \{ \{1\},\{2\} \}$ are both $0$-spheres. 
Unlike for general finite set of sets (sometimes called multi-graphs), 
the cohomology and topology works well for both open and closed sets. Almost everything we
describe here fails for general set of sets which are neither closed nor open. 

\section{Notations}

\paragraph{}
A {\bf finite abstract simplicial complex} $G$ is a finite set of sets, closed
under the operation of taking non-empty subsets. $G$ carries a finite topology $\mathcal{O}$, 
in which the set of all {\bf stars} $U(x) = \{y, x \subset y\}$ is the {\bf topological base}.
If $V=\bigcup_x x$ is the {\bf vertex set}, the collection of vertex stars 
$U(x)$ with $x=\{v\}$ with $v \in V$ are a {\bf topological subbase} which when closed under 
intersection produces the topological base and when closed under
intersection and union produces the topology. The closed
sets in the topology are the sub-simplicial complexes.
If $x \subset y$, then $U(y) \subset U(x)$. A map $f: G \to H$ between simplicial
complexes is {\bf continuous} if $f^{-1}(U)$ is open in $G$ if $U$ is open in $H$. A {\bf simplicial
map} is a map from $V(G)$ to $V(H)$ that lifts to an order preserving map $G \to H$. Simplicial
maps are continuous; however not all continuous maps are simplicial maps. An example is a constant 
map onto a simplex $c$ of positive dimension.

\paragraph{}
The topology on $G$ is {\bf Kolmogorov} (T0) but neither {\bf Fr\'echet} (T1), nor {\bf Hausdorff} (T2). 
As any finite topology, it is {\bf Alexandroff}, meaning that there are smallest 
neighborhoods $U(x)$ of every $x \in G$. These are the {\bf stars} $U(x) = \{ y \in G, x \subset y\}$. 
The topology is {\bf Zariski type} because the closed sets in the topology agreeing with sub-simplicial 
complexes of $G$. In the case when $G$ is the Whitney complex of a graph $(V,E)$, formed by the vertex
sets of complete subgraphs, then subgraphs define closed sets. Every $G$ again defines a graph, where
$\Gamma=(V,E)=\{ (x,y), x \subset y \; , {\rm or} \; , \; y \subset x \}$. 
The Whitney complex $G$ of this graph is the {\bf Barycentric refinement} $G_1$ of $G$. 
The topology has the desired connectivity properties of $G$. 
The just described graph obtained from $G$ has the same connectivity properties than the topology
of $G$. The topology induced from the geodesic distance of the graph would produce the 
{\bf discrete topology} on $G$ and render $G$ completely disconnected. The non-Hausdorff property is
inevitable. In the case when $G$ comes from a graph $\Gamma$, the graph $\Gamma$ 
is the \v{C}ech nerve of the {\bf topological subbase} of vertex stars. 
The \v{C}ech nerve of the {\bf topological base} of all stars is $\Gamma_1$, 
the Barycentric refinement of $\Gamma$. 

\paragraph{}
Elements in $G$ also known as {\bf simplices} or {\bf faces} or simply called {\bf sets}. 
$G$ is is a set of sets of $V$ and the topology is a set of sets in $\mathcal{O}$. 
The closure $\overline{A}$ 
of an arbitrary set $A \subset G$ is the smallest closed set in $G$ containing $A$. We need to
distinguish three different things: (i) the set $x \in G$ as an element or point of $G$, (ii) the subset 
$A(x)=\{x\} \subset G$ and (iii) its closure $K(x)=\overline{\{ x\} } =\{ y, y \subset x \} \subset G$,
which is the simplicial complex generated by $\{x\}$. The set $A(x)$ is open only if $x$ is a locally
maximal simplex (meaning not contained in any larger simplex) and closed only if $x$ has 
dimension $0$. In general, the set $A(x)=\{x\} \subset G$ is neither open nor closed 
but $x$ always defines two natural sets, the {\bf open set} $U(x)$ and the {\bf closed set} $K(x)$. 
The open set $U(x)$ is the smallest open set containing $x$, the closed set $K(x)$ is the smallest 
closed set containing $x$. 
$U(x)$ is the {\bf star} and $K(x)$ the {\bf core}. The closure $B(x)$ of $U(x)$ is the {\bf unit ball}, 
and its boundary $S(x)=B(x) \setminus U(x)$ of $U(x)$ is the {\bf unit sphere}. 
The {\bf dimension} ${\rm dim}(x)$ of $x \in G$ is defined as $|x|-1$, where $|x|$ is the {\bf cardinality} 
of $x$.  We write $X=(x_1,\dots,x_k) \in G^k$ to address a {\bf $k$-tuple} of 
points $x_j \in G$, and define $w(x)=(-1)^{{\rm dim}(x)}$ and $w(X)=\prod_{j=1}^k w(x_j)$. 
We write shorthand $\bigcap X$ for $\bigcap_{j=1}^k x_j$. Similarly,  
we write $\bigcup X$ for $\bigcup_{j=1}^k x_j$. 
We usually do not care about the order of the elements in $X$ but allow that the same element appears
multiple times. The configuration $X=(x,x,x, \dots, x)$ for example is a $k$-point configuration in which 
all points are the same. 

\paragraph{}
We often look at functions $h:G \to R$, where $R$ is some algebraic object like a {\bf ring}.
It should have an additive structure that is commutative (with respect to addition). 
It could be $\mathbb{Z}$ or a finite Abelian group for example, it can also have more structure like a 
vector space over a field or an operator algebra. 
Here we look at rings like $R=\mathbb{Z}$ but the multiplicative structure of the ring does not 
really enter. It could be a division algebra like the quaternions for example, where the multiplication is
not commutative. As we have expressions like $w(x) w_m(U(x))$ in our main result, a multiplication with 
$1$ or $-1$. If $R$ is an additive group, then $-r$ denotes the additive inverse of $r$ in the group. 
As the theme is part of a finitist approach to mathematics, we prefer finite objects. Even in the
case $r=\mathbb{Z}$, the {\bf range} of a function $f: G \to R$ is always finite so that we still deal
with finite objects, despite the fact that there is no a priori cap on the size of $R$.
Having a ring rather than only an additive group has the advantage that on can also look at objects
like determinants (or in the non-commutative ring case {\bf Dieudonn\'e determinants}) which can play a role 
in the vicinity of what we do here. Having a ring is a familiar frame work in other parts of mathematics.

\paragraph{}
Rather than have a fixed ring $R$ it is possible to attach a ring $R(x)$ to every open set $R(U)$.
The ring $R(x)$ is called the {\bf ring of sections}. To fix the relations, one needs {\bf restriction maps}. 
Because the open sets $U(x)$ are minimal, the ring $R(x)$ can be called the {\bf stalk} at $x$. 
Its elements are the {\bf germs} at $x$.  
Since $x \subset y$ implies $U(y) \subset U(x)$, a {\bf sheaf} is determined by giving 
{\bf restriction maps} $r(x,y): U(x) \to U(y)$ satisfying {\bf pre-sheaf properties} $r(x,x) = Id$ and 
$r(y,z) \circ r(x,y)=r(x,z)$ if $x \subset y \subset z$. Having all the germs $R(x)$ and the transition maps 
fixed, one already has the existence (called {\bf gluing}) and uniqueness (called {\bf locality}) 
which are necessary to have a {\bf sheaf} and not only a pre-sheaf. 
An example is $w(x) = (-1)^{{\rm dim}(x)}$ with restriction maps 
$r(x,y) = w(x) w(y)$ for $x \subset y$. There is a unique extension of $w$ to all open sets 
$w(U)=\prod_{x \in U} w(x)$. In analogy to $\chi(U)=w_1(U)=\sum_{x \in U} w(x)$, we have called this
the {\bf Fermi characteristic} $\phi(U)$ of $U$ \cite{KnillEnergy2020}. It is equal to 
${\rm det}(L)$ with the connection matrix $L(x,y)=\chi(K(x) \cap K(y))$, the inverse of the 
Green function matrix $g(x,y) = w(x) w(y) \chi(U(x) \cap U(y))$. 

\paragraph{}
For a general function $h: G^m \to R$, where we write $h(X)=h(x_1, \cdots, x_m)$,
we can define $w_m(G) = \sum_{X \in G^m, \bigcap X \in G} h(X)$ and 
more generally, 
$$  w_m(A) =\sum_{X \in A^m,\bigcap X \in A} h(X) $$ 
for $A \subset G$ for sets $A$ that are not necessarily simplicial complexes.
No symmetry like that $h(X)=h(Y)$ if $X$ and $Y$ are permutations is assumed.
We mostly take $h(X)=\prod_{j=1}^m (-1)^{{\rm dim}(x)}$ because this assures that
$w_m(G)$ are {\bf combinatorial invariants}, meaning invariant under Barycentric refinements. 
For $k$ arbitrary stars $U(x_1),\dots, U(x_m)$, define $U(X)=\bigcap_{j=1}^k U(x_j)$ and 
$\omega_m(U) = \sum_{X \in G^m, \bigcap X \in G} w(U(X))$. 
The first characteristic $w_1(G) = \sum_{x \in G} w(x)$ is the {\bf Euler characteristic} of
$G$, the second characteristic $w_2(G) = \sum_{x,y,x \cap y \in G} w(x) w(y)$ 
is the {\bf Wu characteristic} and the third characteristic is 
$$ w_3(G) =  \sum_{x,y,z, x \cap y \cap z \in G} w(x) w(y) w(z) \; . $$
While for any subsets $A,B$, we have for $m=1$ the property $w_1(A \cup B)=w_1(A) + w_1(B)-w_1(A \cap B)$, 
this valuation formula fails for $m>1$ for general $A,B$, even for closed $A,B$ in general. 
But $w_m(G)$ is a {\bf multi-linear valuation} if extended to $w_m(G_1,\dots,G_m)$ with 
$G_j \subset G$ arbitrary subsets of $G$. 
Now $G_j \to w_m(G_1, \dots, G_m)$ with all other sets $G_i, i \neq j$ fixed, satisfies 
linearity $\omega(A) + \omega(B) = \omega_m(A \cup B) + \omega_m(A \cap B)$.

\paragraph{}
The 1-point complex is defined to be contractible. If $G$ is a complex and both 
$G \setminus U(x)$ and $S(x)$ are contractible, then $G$ is called {\bf contractible}. 
Two complexes $G,H$ which can morphed into each other by homotopy reductions and extensions 
are called {\bf homotopic}. This is an equivalence relation on the space of all complexes. 
Homotopy does not honor dimension and so is not topological. 
Homotopy preserves $w_1(G)$ but not $w_m(G)$ with $m>1$. The higher characteristics are topological
invariants in the sense that they are preserved under homeomorphisms, an other equivalence relation:
$H$ is called a {\bf continuous image} of $G$ if there exists a Barycentric refinement $G_m$ and a continuous
map $f:G_m \to H$ (of course using the finite topology defined above), such that $f^{-1}(S(x))$ is 
homeomorphic to $S(x)$ (inductively defined as the maximal dimension of $S(x)$ is smaller 
than the maximal dimension of $G$) and such that for every locally maximal $x \in G$ of dimension 
$d$, the complex $f^{-1}(B(x))$ is a $d$-ball. $G$ and $H$ are homeomorphic, if $G$ is a 
continuous image of $H$ and $H$ is a continuous image of $G$. 
A $d$-ball is a complex of the form $G-U(x)$, where $G$ is a $d$-sphere. A $d$-sphere
is a $d$-manifold $G$ such that $G-U(x)$ is contractible. A {\bf $d$-manifold} $G$ is a complex such that
for all $x$, the unit sphere $S(x)$ is a $(d-1)$-sphere. The empty complex is the $(-1)$-sphere. 

\section{Energy theorem}

\paragraph{}
We assume $m,k \geq 1$ are integers. Define the $m$'th order
{\bf potential energy} of the {\bf $k$-point configuration} $X \in G^k$ as
$g_m(X) = w(X) w_m(U(X))$. By design, it is zero if $U(X)=\bigcap U(x_j) = \emptyset$
which is the case if at least one of the points is out of reach of the others. 
Even for non-intersecting $x,y$ it can happen that $U(x) \cap U(y)$ is non-empty
like if $x,y$ are zero-dimensional parts of a higher dimensional simplex $z$, where
$z$ is in $U(x) \cap U(y)$. This is a case where $x,y$ can not be separated by open
sets. Our goal is to have for any $k \geq 1$ and any $m \geq 1$, the $m$'th characteristic
can be expressed using {\bf $k$-point Green function entries}
$g_m(x_1, \dots, x_k) = \prod_{j=1}^k w(x_j) w_m(\bigcap_{j=1}^k U(x_j))$. We think of this
as the {\bf $m$-th potential energy} of the $k$-point configuration $X=(x_1,\dots, x_k)$. 

\paragraph{}
Our main theorem tells that the $m$'th characteristic is the {\bf total energy} over 
all possible $k$'point configurations. 

\begin{thm}[Energy]
$w_m(G) = \sum_{X \in G^k} g_m(X)$.
\end{thm}

\begin{proof}
The theorem is simpler to prove if formulated more generally like if 
$$ h_m(x_1, \dots, x_m) = \prod_{j=1}^m w(x_j) $$ 
is replaced by a general function $h_m(x_1,\dots, x_m)$
of $m$ variables. The reason is that the map $h_m \to g_k$ is linear so that we only need to verify the 
statement in the simplest possible case where $h$ is $1$ only for a single
configuration $Z=(z_1, \dots, z_m)$ and $0$ else. The left hand side is then $1$. 
On the right hand side, we have to look at all 
$X=(x_1,\dots, x_k)$ for which $Z \in U(X)=U(x_1) \cap U(x_2) \cap \cdots \cap U(x_m)$. 
As we will see however below, that condition will assure that $\bigcup Z \subset U(x_j)$ 
for all $j=1,\dots ,m$. Therefore, 
$\sum_{X \in G^k} w(X) w_m(U(X)) = \prod_{j=1}^k \sum_{x_j \in G} w(x_j) w_m(U(x_j)) =1$. 
\end{proof} 

\paragraph{}
A second major point is the {\bf sphere formula} for $S(X) = \delta U(X)=B(x) \setminus U(x)$. 

\begin{thm}[Sphere formula]
$0 = \sum_{X \in G^k} w(X) w_m(S(X))$. 
\end{thm}

This means that the total energy of all unit spheres of even configurations
is the same than the total energy of all unit spheres of odd configurations. 

\begin{proof}
The energy theorem also works of $U(X)$ is replaced by $B(X)$, which is the 
closure of $U(X)$.  The two equations 
$$ 0 = \sum_{X \in G^k} w(X) w_m(U(X)) $$
$$ 0 = \sum_{X \in G^k} w(X) w_m(B(X)) $$
and the local valuation formula
$$   w_m(U(X)) - (-1)^m w_m(S(X)) = w_m(B(X)) \;  $$
prove the theorem. 
\end{proof}

\paragraph{}
Let us add already a remark which however leads to an other story to which we hope to be able to
write more about in the future: 
the theorem also works for the {\bf dual spheres} $\hat{X}=S(x_1) \cap \cdots \cap S(x_k)$ which is 
different from $S(X) = \delta (U(x_1) \cap \cdots \cap U(x_k))$. 
$$ 0 = \sum_{X \in G^k} w(X) w_m(\hat{X})  \; . $$

\paragraph{}
The dual spheres $\hat{X}$ play an important role in other places like in {\bf graph coloring}. 
In the case when $G$ is a $d$-manifold and if all points in $X$ are adjacent in
the  metric in which two points $x,y \in S(x)$ have distance $1$, then $\hat{X}$ is always 
a $(d-k)$-sphere. If $k=1$ then $\hat{X} = S(x)$, which is in the manifold case a $(d-1)$-sphere.
For $k=2$, and  $x,y \in S(x)$ we have $\hat{(x,y)} = S(x) \cap S(y)$ is a $(d-2)$-sphere. 
We have looked at this earlier in the context of graphs, where $X=(v_1,\dots,v_k)$ are the vertices
of a complete graph $K_k$ and so is associated to a vertex $x$ in the Barycentric refined graph. 
If $G$ is a $d$-manifold, then by definition $\hat{X}=\hat{x}$ is a $(d-k)$-sphere. We have seen
this as a {\bf duality} because $\bigcap_{v \in x} S(v)=\hat{x}$ and 
$\bigcap_{v \in \hat{x}} S(v) = x$. This can be seen as a duality between $(k-1)$-spheres 
(the boundary sphere of a simplex) and $(d-k)$-spheres. 
Shifting $k$, there is a duality between $k$-spheres and $(d+1-k)$-spheres.  

\paragraph{}
Still dwelling on that, moving cohomology from simplices to spheres could allow then to see {\bf Poincar\'e duality}
more elegantly, avoiding concepts like CW complexes which are necessary already in elementary
setups like that the cube is the dual of the octahedron. When the cube is seen as a 2-sphere, it 
by definition is not a simplicial complex but only a more general CW complex if one wants 
to understand it as dual to the octahedron. So,
if one wants to avoid the continuum (as we do), one usually goes to CW complexes. Poincar\'e himself
already struggled quite a bit with the difficulty that the dual of a simplicial complex is not a simplicial
complex. We will see the duality in the context of $\delta$-sets which generalize simplicial complexes
too but more naturally than CW complexes. $\delta$ sets are more general than simplicial sets, a construct
which has a bit more structure than $\delta$ sets. But simplicial sets as a subclass of $\delta$ sets 
less accessible. Entire articles have been written just to explain the definition. 

\paragraph{}
Here is an other remark, maybe more for the mathematical physics minded: 
motivated by the {\bf Fock picture} in which $G^k$ are considered as 
$k$-particle configurations, we could sum up the theorem over $k$ and get for example
$$ w_m(G) = \sum_{k=1}^{\infty} \frac{1}{2^k} \sum_{X \in G^k} g_m(X) \; . $$
If we would replace $w(x)$ with $h(x)=(-1)^{{\rm dim}(x)}/2$ and set $h(X)=\prod_j h(x_j)$ and
still use $U(X)=\bigcap_j U(x_j)$, if $X=(x_1,\dots,x_k)$
and setting $g_m(X) = h(X) w_m(U(X))$, this would allow to write
$w_m(G) = \sum_{X} g_m(X)$ where $X$ runs over all particle configurations ranging 
from single particles $k=1$ to pair interactions $k=2$, 
triple interactions $k=3$ etc. The total {\bf energy of space} is then the 
sum over all interaction energies overall  possible particle configurations $X$. 
If we think of $g_m(X)$ as the {\bf potential energy} of the particle configuration $X$, 
then $w_m(G)$ is the sum over all possible potential energies which particle configurations 
which can be realized in $G$. 

\paragraph{}
Here are special cases, written out in more detail. First of all, lets write down the definitions
of {\bf Euler characteristic}
$$   w_1(G) = \sum_{x \in G} w(x) $$ 
and {\bf  Wu characteristic} 
$$ w_2(G) = \sum_{x,y \in G^2, x \cap y \neq \emptyset} w(x) w(y) \; . $$

i) For $m=1$, we have expressions for the Euler characteristic: \\
$w_1(G) = \sum_{x} w(x) w_1(U(x)) = \sum_x g_1(x)$ \\
$w_1(G) = \sum_{x,y} w(x) w(y) w_1(U(x) \cap U(y)) = \sum_{x,y} g_1(x,y)$ \\
$w_1(G) = \sum_{x,y,z} w(x) w(y) w(z) w_1(U(x) \cap U(y) \cap U(z)) = \sum_{x,y,z} g_1(x,y,z)$ \\
$w_1(G) = \sum_{x,y,z,w} g_1(x,y,z,w)$ \\
The first two expressions for $k=1$ and $k=2$ have appeared in \cite{KnillEnergy2020}.
The inverse of the matrix $g_1(x,y)$ is $L(x,y) = 1$ if $x =y$ and $L(x,y)=0$ if $x \neq y$.   \\
ii) Next come expressions for the Wu characteristic: \\
$w_2(G) = \sum_{x} w(x) w_2(U(x)) = \sum_x g_2(x)$ \\
$w_2(G) = \sum_{x,y \in G} w(x) w(y) w_2(U(x) \cap U(y)) = \sum_{x,y} g_2(x,y)$ \\
$w_2(G) = \sum_{x,y,z \in G} w(x) w(y) w(z) w_2(U(x) \cap U(y) \cap U(z)) = \sum_{x,y} g_2(x,y,z)$ \\
$w_2(G) = \sum_{x,y,z,w} g_2(x,y,z,w)$ \\
iii) And here are expressions for the third characteristic \\
$w_3(G) = \sum_{x} w(x) w_3(U(x)) = \sum_x g_3(x)$ \\
$w_3(G) = \sum_{x,y} w(x) w(y) w_3(U(x) \cap U(y)) = \sum_{x,y} g_3(x,y)$ \\
$w_3(G) = \sum_{x,y,z} w(x) w(y) w(z) w_3(U(x) \cap U(y)) \cap U(z)) = \sum_{x,y,z} g_3(x,y,z)$ \\
$w_3(G) = \sum_{x,y,z,w} g_3(x,y,z,w)$ \\ 

{\bf Remarks}:
{\bf 1)} The theorem works if $h(X)=w(X)$ is replaced by any $R$-valued interaction 
function $h: G^k \to R$. We would for example set $w_2(A) = \sum_{x,y, x \cap y \in A} h(x,y)$ and
get this equal to $\sum_{x,y \in G^2} g_2(x,y)$ with 
$$  g_2(x,y) = w(x) w(y) w_2(U(x) \cap U(y)) \;  $$
or equal to $\sum_{x \in G} g_1(x)$ with $g_1(x) = w(x) w_2(U(x))$. \\
{\bf 2)} The energy theorems in the case of $k=1$ are of Gauss-Bonnet type because we attach
a fixed curvature $g_m(x)$ to a point. We can actually interpret this also as a Poincar\'e-Hopf
theorem. There is some duality here as $U(x)$ can be seen dual to $K(x)$. \\
{\bf 3)} The energy formula reduces the time for the computation of the characteristic 
substantially. Especially for $k=1$, where we have only to compute the 
$m$'th characteristic of the $n=|G|$ stars $U(x)$. \\
{\bf 4)} Instead of self-interactions of all $G^j$, we could take open sets 
$G_j \subset G$ and get expressions 
$$w_m(G_1,\dots, G_k) = \sum_{X, x_j \in G_j} g_m(X) \; . $$
For example, $w_m(A,B,C) = \sum_{x \in A, y \in B, z \in C} g_m(x,y,z)$,
where 
$$  g_m(x,y,z) = w(x) w(y) w(z) w_m(U(x) \cap U(y) \cap U(z)) \; . $$
Think of this as the {\bf total m'th characteristic energy} of the three sets $A,B,C$.  \\
{\bf 5)} We have seen that in order to prove the theorems, it is helpful 
to energize more generally $h(x_1,\dots,x_m)$ and set 
$w_m(G)=\sum_{X \in G^m, X>0} h(X)$. The Green function procedure which maps
functions $h: G^k \to R$ to functions $g: G^m \to R$ with $g_m(X) = w(X) w_m(U(X))$
is linear. It is therefore enough to verify the claim for basis elements,
where h(x) is zero except at $h(x_1,\dots,x_k)=1$.

\section{Valuation  Lemma}

\paragraph{}
The following theorem is a main reason, why higher characteristics 
are topological invariants. The following valuation formula does not work for
closed sets in general already if $m \geq 2$. 

\begin{lemma}[Valuation] 
$w_m(U \cup V) = w_m(U) + w_m(V) - w_m(U \cap V)$ for all open $U,V$ and all $m \geq 1$. 
\end{lemma}

\paragraph{}
The key why this works is: 

\begin{lemma}[Patching] 
Given $X \in G^k$ and $U,V$ are open, define $z=\bigcap X$.
If $z \in U \cap V$, then both $x,y$ are simultaneously in $U,V$ and so 
$x \in U \cap V$ and $y \in U \cap U$. 
\end{lemma}

\begin{proof}[Patching]
If $z=\bigcap X$ is in $U \cap V$, then every simplex $w$ which contains $z$
is both in $U$ and $V$ and so in $U \cap V$ so that $w$ is in $U$ as 
well as in $V$ and $U \cap V$. Because this works for any $w$, it works for 
any of the points $x_j$. Since all $x_j $are in $U,V,U \cap V$ also 
the union $z = \bigcup X$ is in all three sets. 
\end{proof} 

\begin{proof}[Valuation]
Counting $X$ in $U \cup V$ is the sum of the counts in $U$ and the counts in $V$
with a double count if $X$ is in the intersection. 
\end{proof} 

{\bf Remark:}  \\
{\bf 1)} For $m>1$, this in general does not hold for closed sets nor for 
a mixture of closed and open sets if $m>1$. We will see for general $m$ that
for an open set $A=U(x)$, a closed disjoint set $B=S(x)$ and closed union $B(x)$
and empty intersection we have $\omega_m(U(x)) - \omega_m(S(x))= \omega_m(B(x))$ but
for odd $m$ that  $\omega_m(U(x)) + \omega_m(S(x))= \omega_m(B(x))$. \\
{\bf a)} For $G=\{ (1),(2),(3),(12),(23)\}$, the two closed sets  \\
$A=\{(1),(2),(12)\}, B=\{(2),(3),(23)\}$ intersect in $C=\{ (2) \}$. 
Now $w_2(A)=w_2(B)=-1$ now $\omega_2(A) + \omega_2(B)=w_2(G)-\omega_2(S)$. \\
{\bf b)} If $G$ is the octahedron complex which is the join 
$C_4 \oplus \{a,b\}$ 
and $A=G \setminus U(a), B=G \setminus U(b)$ which are both balls 
intersecting in
the circle $S=\{x_1,x_2,x_3,x_4,(x_1 x_2),(x_2 x_3),(x_3 x_4),(x_4 x_1) \}$ 
which is also closed. The pair $(a x_1) \subset A$ intersects with the pair 
$(b x_1) \subset B$ but these two pairs are both not in $S$. Still 
$w_1(A) + w_2(B)=w_2(G) + w_2(S)$. \\
Example a) shows that looking for more general relations is rather 
pointless as it depends on $m$ and the dimension.  \\
{\bf 2)} In the special case $m=1$, where we deal with Euler characteristic, 
the valuation formula holds for all sets
because there is no interaction between points. This is the reason
that for $m=1$, we have a homotopy invariant while for $m>1$ we have 
topological invariants.

\paragraph{}
The local valuation formulas link two closed and an open set:

\begin{lemma}[Local valuation]
$$ w_m(B(x)) = w_m(U(x)) - (-1)^m w_m(S(x)) \;  $$
\end{lemma}

This works more generally for the open sets $U(X)=\bigcap U(x_j)$ and 
$B(X)=\overline{U(X)}$ and $S(X) = B(X) \setminus U(X)$, the boundary. 
The reason is that $U(X)$ is a disjoint union of disjoint open $U(x)$. 

\section{The case of manifolds}

\paragraph{}
For $x \in G$, the {\bf unit ball} $B(x) = \overline{U(x)}$ and the 
{\bf unit sphere} $S(x)=B(x) \setminus U(x)$ are both closed sets and 
so simplicial complexes. A complex $G$ is called a {\bf $d$-manifold} if every $S(x)$ is 
a $(d-1)$-sphere. A $d$-manifold is a {\bf $d$-sphere}, if there exists 
$x$ such that $G \setminus U(x)$ is contractible. 
In the case when $G$ is a $d$-sphere, $G \setminus U(x))$ is declared to be a 
{\bf $d$-ball}. A complex $G$ is called {\bf contractible} if there exists $x \in G$ 
such that both $S(x)$ and $G \setminus U(x)$ are contractible. These definitions are recursive with 
respect to the dimension $d$. The foundation is laid by assuming the {\bf void} 
$0=\emptyset$ is the $(-1)$-sphere and that the one-point complex called 
{\bf unit} $1=\{1\}$ is the smallest contractible complex. 

\paragraph{}
A {\bf $d$-manifold with boundary} is a complex $G$ such that every 
$S(x)$ is either a $(d-1)$-sphere or a $(d-1)$ ball. The {\bf boundary} $\delta G = 
\{ x \in G, S(x) \; {\rm ball} \}$ is a closed subset of $G$. 
(Just check that if $x \in \delta G$ and $z \subset x$ then $S(z)$ is a ball).
It is a $(d-1)$ manifold because for $z \in \delta G$ the set $S(z) \cap \delta G$
is a sphere because $B(z) \cap \delta G$ is a $(d-1)$ unit ball. We have to distinguish
here between ``closed= no boundary" and ``closed=topologically closed".  
We also want to deal with the case of open manifolds without boundary like an open disk or
topologically closed manifolds with boundary which are not closed as manifolds. 
Open manifolds model also infinite manifolds like in the continuum, where for example the 
contangent correspondence allows to see the space $(-1,1)$ to be naturally 
homeomorphic to $\mathbb{R}$. Let us use the notation $U$ for an open manifold without boundary, 
$G$ for the closure and $S$ the boundary. So, $S=\delta G$ is the topological boundary.  \\

\paragraph{}
We can verify $w_m(B(x)) = w_m(U(x)) -(-1)^m  w_m(S(x))$ in the case of manifolds
without boundary. 

\begin{lemma}[Local data in manifold case] 
$w_m(U(x))=(-1)^{d m}$, $w_m(S(x)) = 1+(-1)^{d-1}$ and $w_m(B(x))=(-1)^{d(m+1)}$.
\end{lemma}



\begin{thm}[Manifolds $M$]
(i) For even-dimensional manifolds with or without boundary 
$w_{m}(M) = w_{1}(M)$. 
(ii) For odd-dimensional manifolds with boundary 
$w_m(M)=w_1(M)-(-1)^m w_1(S)$. 
\end{thm}
\begin{proof}
For even dimensional manifolds with boundary $S$,
we have $w_m(U(x))=1$ in the interior and $w_m(U(x))=0$ at 
the boundary, so that the formula $w_m(M) = \sum_x \omega(x) w_m(U)$ 
immediately shows that $w_m(M) = \sum_x w(x) = w_1(M)$.
This also works for odd-dimensional manifolds without boundary. 
For odd-dimensional manifolds with boundary $S$, we 
have $w_m(B(x))=1$ in the interior and $w_m(B(x))=-(-1)^m$
at the boundary. So $w_m(M)=w_1(M)-(-1)^m w_1(S)$ with or without boundary. \\
In the special case if the manifold has no boundary, we also can use
the energy theorem $w_m(G) = \sum_x w(x) w_m(U(x))$, 
The previous lemma $w_m(U(x))=w(x)^m =(-1)^{dm}$ shows that
if $d$ is even, then $w_m(U(x))=1$ and so
$w_m(G)=\sum_x w(x) = w_1(G)$. If $d$ is odd, then
$w_1(G)= \sum_x w(x) = 0$ and this does not change
when multiplying with the constant $w_m(x)=(-1)^{d m}$.
\end{proof} 

\section{Topological invariance}

\paragraph{}
A simplex $x$ is called {\bf locally maximal} if it is not contained in any 
larger simplex. This means that the star $U(x) = \{ x\}$ of a locally 
maximal point therefore has only one point $x$ in $G$. 
A map $f: G \to H$ between simplicial complexes is called {\bf continuous} 
if $f^{-1}(U)$ is open in $G$ if $U$ is open in $H$.  
A complex $H$ is a {\bf topological image} of $G$
if there exists a continuous $f:G \to H$ such that $f^{-1}(S(x))$ is 
homeomorphic to $S(x)$ for all $x$ and $f^{-1}(B(x))$ is a $dim(x)$-ball 
for every unit ball $B(x)$ with locally maximal $x$ in $H$. 
If $G$ and $H$ are both topological images of each other, the two 
complexes are called {\bf homeomorphic}. 
Also this definition is recursive. It relies on homeomorphism in 
dimension $(d-1)$. Quantities which are constant on homeomorphism
classes are called {\bf topological invariants}.

\paragraph{}
In comparison, a quantity is a {\bf combinatorial invariant} 
if it stays the same under {\bf Barycentric refinements}.
We have in this context say what we mean with the Barycentric refinement
of an open set. We can not just take the Whitney complex of the incidence
graph as in the case of closed sets because the Whitney complex is a 
simplicial complex. With that definition, the Barycentric refinement of an
open set would be a closed set. What we can do is to see an open set $U$
of a complex $G$ and define $U_1$ as the complement of the 
Barycentric refinement $(G \setminus U)_1$ in the Barycentric refinement 
$G_1$ of $G$.  Bott defined {\bf combinatorial invariant} as a quantity that
does not change when making Barycentric refinements. For all $m \geq 1$ we have:

\begin{thm}[Invariance]
$w_m$ are topological invariants.
$w_m$ are combinatorial invariants.
\end{thm}

\begin{proof}
First check that $w_m$ is the same on all d-balls $B$. This follows from the 
general manifold formula because $B$ is a manifold with boundary. For odd $m$
we have $w_m(B)=w_1(B)$. For even $m$, we have $w_m(B)=w_1(B)-w_1(S)$, where $S$
is the boundary. The fact that $w_m$ is invariant under Barycentric refinement
is treated in the next section about the topological product. 
Given two general $G,H$ which are homeomorphic. By definition, there is
a continuous map from $G_n$ to $H$.
Now use the Ball formula $w_m(H) = \sum_x w(x) w_m(B(x))$. 
We claim that this is the same than $w_m(G) = \sum_{f^{-1}(B(x))} w(x) w_m(f^{-1}(B(x)))$
\end{proof}

\paragraph{}
The property of being a {\bf Dehn-Sommerville space} (as defined above again) is
topological. If $G,H$ are homeomorphic and $G$ is Dehn-Sommerville,
then $H$ is Dehn-Sommerville. The proof can be done by induction. 
Having verified it for dimension up to $d$, then we have  it for 
dimension $d+1$ because there is a topological correspondence of 
unit spheres making sure that all unit spheres are Dehn-Sommerville. 

\paragraph{}
As an other question we should mention how to define homotopy for open 
sets.  One idea two call two open sets $U,V$ to be
homotopic in $G$, if $U^c = G \setminus U$ and $V^c$ are homotopic closed sets.
But we would like to have that if $U,V$ are homotopic, then the Betti vectors 
$b(U)$ and $b(V)$ should be the same. An open 3-ball for example has Betti vector 
$(0,0,0,1)$ and a closed 3-ball has Betti vector $(1,0,0,0)$. They should already not
be considered homotopic because the Euler characteristic does not match. All this
will hopefully be explored more in a future paper on the matter. 

\section{Topological product}

\paragraph{}
The {\bf topological product} $G \cdot H$ of two simplicial complexes $G,H$ is the Whitney 
complex of the graph in which the Cartesian product $G \times H$ are the vertices and where $(a,b)$
and $(c,d)$ are connected either $a \subset b, c \subset d$ or
if $b \subset a, d \subset c$. It is a set of subsets of $V=G \times H$ and again 
a simplicial complex. 

\paragraph{}
Unlike other products like the {\bf Shannon product}, the topological product is of a 
topological nature. It preserves manifolds and as we will see is compatible with all 
higher characteristics. 

\begin{lemma}
If $G$ and $H$ are manifolds, then $G \cdot H$ is a manifold. 
\end{lemma}

\begin{proof}
Lets outline again the argument:
we have to look at the unit sphere $S( (x,y) )$ of a point $(x,y)$ 
and show that it is a $d-1$-sphere. 
$S( (x,y) )$ is the union of two cylinders $S(x) \cdot B(y)$ 
and $B(x) \cdot S(y)$ which are glued together along 
$S(x) \cdot S(y)$. In the case of a product of a $1$-manifold  $G$
and a $1$-manifold $H$ for example, the unit spheres $S(x,y)$ is a
"square" which is the union of $S(x) \times B(y)$ (left right parts)
$B(x) \times S(y)$ (top bottom part) intersecting in $4$ points
$S(x) \times S(y)$. To see that this is a sphere, collapse $B(y)$ 
to a point to see that $S(x,y)$ is homotopic to the join of $S(y)$
and $S(x)$ which is a sphere. One then checks case by case that each 
point in $S(x,y)$ has a $d-2$ sphere as unit sphere. For interior
points in $B(y)$ or $B(x)$ this follows by induction. For 
points in $S(x) \times S(y)$ the unit sphere is a copy of two $(d-2)$
balls glued along a $(d-3)$ sphere and so a $(d-2)$ sphere. 
\end{proof} 

\paragraph{}
Similarly, one has by analyzing additionally the unit spheres of boundary 
points and knowing that the join of a ball with a sphere is a ball:

\begin{lemma}
If $G,H$ are manifolds with boundary then $G \cdot H$ is a 
manifold with boundary. 
\end{lemma}

For example, the product of two closed intervals is a square, the product
of an interval with a circle is a closed cylinder. 

\paragraph{}
The topological product is not associative. 
The complex $G \cdot 1$ is the Barycentric refinement $G_1$ of $G$ so that
$G \cdot (1 \cdot 1)=G_1$ but $(G \cdot 1) \cdot 1 = G_2$ is the 
second Barycentric refinement. 

\paragraph{}
There is an algebraic {\bf Stanley-Reisner} description (which is seen implemented in the 
code section). 
If the elements in $V=\bigcup_{x \in G} x$ are labeled with variables $t_1,\dots,t_q$, 
then every $x \in G$ can be written as a monomial $t(x)=t_{j_0} t_{j_2} \dots t_{j_{{\rm dim}(x)}}$ 
and $x \subset y$ is algebraically encoded by $t(x) | t(y)$. The complex $G$ is a polynomial 
$\sum_{x \in G} t(x)$. In the same way, the complex $H$ is given by the polynomial $\sum_{y \in H} t(y)$.
The product $G \cdot H$ is now represented by the product of these two polynomials. 
The vertex set of $G \cdot H$ has $|G| |H|$ elements. 

\paragraph{}
For example, if $G=\{\{1\},\{2\},\{1,2\}\}$ and $H=\{ \{3\},\{4\},\{3,4\} \}$.
The product complex $G \cdot H$ is a complex for which the vertex set has $9$ 
elements. 

\begin{thm}[Product]
$w_m(G \cdot H)=w_m(G) w_m(H)$.
\end{thm}

\begin{proof}
We have $w_m(G) = \sum_{x_1,\dots,x_m, \bigcap X>0} h(x_1) \cdots h(x_m) = \sum_X h(X)$. 
Similarly, $w_m(H) = \sum_{y_1,\dots,y_m, \bigcap Y>0} h(y_1) \cdots h(y_m) = \sum_Y h(Y)$.
Now, 
\begin{eqnarray*}
   w_m(G \cdot H) &=& \sum_{X,Y\bigcap X>0,\bigcap Y>0}  
    h(x_1) \cdots h(x_m) h(y_1) \cdots h(y_m) \\
    &=& \sum_{X,Y} h(X) h(Y)  \\
    &=& \sum_X h(X) \sum_Y h(Y) = w_m(X) w_m(Y)  \; . 
\end{eqnarray*}
\end{proof}

\paragraph{}
{\bf Remarks:} \\
{\bf 1)} A special case is if $H=1=\{\{1\} \}$, where $G_1=G \cdot 1$ is the 
Barycentric refinement. And $w_m(G_1)=w_m(G)$ is a special important case. \\
{\bf 2)} As in the case $m=1$, if $G$ is the Whitney complex of a finite simple
graph, we would like to know about the behavior of the curvature when taking 
products. In the case of the Shannon product and $m=1$, we have seen that the 
curvatures multiply. 

\vfill
\pagebreak

\section{Code}

\paragraph{}
Here is some code. As usual, one can copy-paste it from the ArXiv's 
LaTex source. The procedures should be pretty self-explanatory,
given that the Wolfram language allows to write mathematical procedures in a form 
resembling pseudo code. As for the topological product, we repeat the algebraic
frame work as we have used it \cite{KnillKuenneth}, (before even knowing about 
Stanley-Reisner rings). 

\begin{tiny}
\lstset{language=Mathematica} \lstset{frameround=fttt}
\begin{lstlisting}[frame=single]
Closure[A_]:=If[A=={},{},Delete[Union[Sort[Flatten[Map[Subsets,A],1]]],1]];
Boundary[A_]:=Complement[Closure[A],A];
Whitney[s_]:=If[Length[EdgeList[s]]==0,Map[{#}&,VertexList[s]],
   Map[Sort,Sort[Closure[FindClique[s,Infinity,All]]]]];
UU[G_,x_]:=Module[{U={}},
    Do[If[SubsetQ[G[[k]],x],U=Append[U,G[[k]]]],{k,Length[G]}];U];
VV[G_,x_]:=Module[{U={}},
    Do[If[SubsetQ[x,G[[k]]],U=Append[U,G[[k]]]],{k,Length[G]}];U];
Basis[G_]:=Table[UU[G,G[[k]]],{k,Length[G]}]; Stars=Basis; 
Cores[G_]:=Table[VV[G,G[[k]]],{k,Length[G]}]; 
SubBasis[G_]:=Module[{V=Union[Flatten[G]]},Table[UU[G,{V[[k]]}],{k,Length[V]}]];
UnitSpheres[G_]:=Module[{B=Basis[G]},
   Table[Complement[Closure[B[[k]]],B[[k]]],{k,Length[B]}]];
UnitBalls[G_]:=Map[Closure,Basis[G]];
Cl[U_,A_]:=Module[{V=U},Do[V=Union[Append[V,
   Union[V[[k]],A[[l]]]]],{k,Length[V]},{l,Length[A]}];V];
Topology[G_]:=Module[{V=B=Basis[G]},
   Do[V=Cl[V,B],{Length[Union[Flatten[G]]]}];Append[V,{}]];
GraphBasis[s_]:=Basis[Whitney[s]];
Nullity[Q_]:=Length[NullSpace[Q]];  sig[x_]:=Signature[x]; 
sig[x_,y_]:=If[SubsetQ[x,y]&&(Length[x]==Length[y]+1),
  sig[Prepend[y,Complement[x,y][[1]]]]*sig[x],0];
Fvector[G_]:=If[Length[G]==0,{},Delete[BinCounts[Map[Length,G]],1]];
Ffunction[G_,t_]:=Module[{f=Fvector[G],n},Clear[t]; n=Length[f];
  If[Length[VertexList[s]]==0,1,1+Sum[f[[k]]*t^k,{k,n}]]];
BarycentricGraph[s_]:=ToGraph[Whitney[s]]; 
BarycentricComplex[G_]:=Whitney[ToGraph[G]]; 
dim[x_]:=Length[x]-1; w[x_]:=(-1)^dim[x];  
Wu1[A_]:=Total[Map[w,A]]; Chi=Wu1;                        EulerChi=Wu1;
Wu2[A_]:=Module[{a=Length[A]},Sum[x=A[[k]]; Sum[y=A[[l]];
   If[MemberQ[A,Intersection[x,y]],1,0]*w[x]*w[y],{l,a}],{k,a}]];Wu=Wu2;
Wu3[A_]:=Module[{a=Length[A]},Sum[x=A[[k]];Sum[y=A[[l]];Sum[z=A[[o]];
   If[MemberQ[A,Intersection[x,y,z]],1,0]*
   w[x]*w[y]*w[z],{o,a}],{l,a}],{k,a}]];
RingFromComplex[G_,a_]:=Module[{V=Union[Flatten[G]],n,T,U},
    n=Length[V];Quiet[T=Table[V[[k]]->a[[k]],{k,n}]];
    Quiet[U=G /.T]; Sum[Product[U[[k,l]],
    {l,Length[U[[k]]]}],{k,Length[U]}]];
ComplexFromRing[f_]:=Module[{s,ff},s={}; ff=Expand[f];
    Do[Do[If[Denominator[ff[[k]]/ff[[l]]]==1 && k!=l,
      s=Append[s,k->l]], {k,Length[ff]}],{l,1,Length[ff]}];
    Whitney[UndirectedGraph[Graph[Range[Length[ff]],s]]]];
TopologicalProduct[G_,H_]:=Module[{},
    f=RingFromComplex[G,"a"];
    g=RingFromComplex[H,"b"]; F=Expand[f*g]; ComplexFromRing[F]];
\end{lstlisting}
\end{tiny}

\pagebreak

\paragraph{}
The following lines illustrate some of the identities for random 
complexes. 

\begin{tiny}
\lstset{language=Mathematica} \lstset{frameround=fttt}
\begin{lstlisting}[frame=single]
G=Whitney[RandomGraph[{9,15}]]; G = Sort[G]; G = Map[Sort, G]; n = Length[G];
U1=Basis[G]; U2=Table[Intersection[U1[[k]],U1[[l]]],{k,n},{l,n}];
U3=Table[Intersection[U1[[k]],U2[[l,m]]],{k,n},{l,n},{m,n}];
V1=Cores[G]; V2=Table[Intersection[V1[[k]],V1[[l]]],{k,n},{l,n}];
V3=Table[Intersection[V1[[k]],V2[[l,m]]],{k,n},{l,n},{m,n}];
S1=Map[Boundary,U1]; S2=Table[Boundary[U2[[k,l]]],{k,n},{l,n}];  
S3=Table[Boundary[U3[[k,l,m]]],{k,n},{l,n},{m,n}];
B1=Map[Closure,U1];  B2=Table[Closure[U2[[k,l]]],{k,n},{l,n}];   
B3=Table[Closure[U3[[k,l,m]]],{k,n},{l,n},{m,n}];

Print[" Check energy formulas   "];
Wu1[G]==Total[Table[w[G[[k]]]*Wu1[U1[[k]]],{k,n}]];
Wu2[G]==Total[Table[w[G[[k]]]*Wu2[U1[[k]]],{k,n}]];
Wu1[G]==Total[Flatten[Table[w[G[[k]]]*w[G[[l]]]*Wu1[U2[[k,l]]],{k,n},{l,n}]]]
Wu2[G]==Total[Flatten[Table[w[G[[k]]]*w[G[[l]]]*Wu2[U2[[k,l]]],{k,n},{l,n}]]]
Wu1[G]==Total[Flatten[Table[w[G[[k]]]*w[G[[l]]]*w[G[[m]]]*Wu1[U3[[k,l,m]]],
  {k,n},{l,n},{m,n}]]]
Wu2[G]==Total[Flatten[Table[w[G[[k]]]*w[G[[l]]]*w[G[[m]]]*Wu2[U3[[k,l,m]]],
  {k,n},{l,n},{m,n}]]]

Print[" Check energy ball formulas  "];
Wu1[G]==Total[Table[w[G[[k]]]*Wu1[B1[[k]]],{k,n}]]
Wu2[G]==Total[Table[w[G[[k]]]*Wu2[B1[[k]]],{k,n}]]
Wu1[G]==Total[Flatten[Table[w[G[[k]]]*w[G[[l]]]*Wu1[B2[[k,l]]],{k,n},{l,n}]]]
Wu2[G]==Total[Flatten[Table[w[G[[k]]]*w[G[[l]]]*Wu2[B2[[k,l]]],{k,n},{l,n}]]]
Wu1[G]==Total[Flatten[Table[w[G[[k]]]*w[G[[l]]]*w[G[[m]]]*Wu1[B3[[k,l,m]]],
  {k,n},{l,n},{m,n}]]]
Wu2[G]==Total[Flatten[Table[w[G[[k]]]*w[G[[l]]]*w[G[[m]]]*Wu2[B3[[k,l,m]]],
  {k,n},{l,n},{m,n}]]]

Print[" Check sphere formulas          "];
0==Total[Table[w[G[[k]]]*Wu1[S1[[k]]], {k, n}]]
0==Total[Table[w[G[[k]]]*Wu2[S1[[k]]], {k, n}]]
0==Total[Flatten[Table[w[G[[k]]]*w[G[[l]]]*Wu1[S2[[k,l]]],{k,n},{l,n}]]]
0==Total[Flatten[Table[w[G[[k]]]*w[G[[l]]]*Wu2[S2[[k,l]]],{k,n},{l,n}]]]
0==Total[Flatten[Table[w[G[[k]]]*w[G[[l]]]*w[G[[m]]]*Wu1[S3[[k,l,m]]],
   {k,n},{l,n},{m,n}]]]
0==Total[Flatten[Table[w[G[[k]]]*w[G[[l]]]*w[G[[m]]]*Wu2[S3[[k,l,m]]],
   {k,n},{l,n},{m,n}]]]

Print[" Check local valuation formula  "];
Map[Wu1,U1]==Map[Wu1,B1]-Map[Wu1,S1]
Union[Flatten[Table[Wu1[U2[[k,l]]]+Wu1[S2[[k,l]]]-Wu1[B2[[k,l]]],{k,n},{l,n}]]]
Map[Wu2,U1]==Map[Wu2,B1]+Map[Wu2,S1]
Union[Flatten[Table[Wu2[U2[[k,l]]]-Wu2[S2[[k,l]]]-Wu2[B2[[k,l]]],{k,n},{l,n}]]]

Print[" Check product  "];
G=Whitney[RandomGraph[{6,10}]]; H=Whitney[StarGraph[5]];
GH=TopologicalProduct[G,H]; 
{Wu1[G],Wu1[H],Wu1[GH]}
{Wu2[G],Wu2[H],Wu2[GH]}

(* A bit more time consuming to compute *)
Wu3[G]==Total[Table[w[G[[k]]]*Wu3[B1[[k]]],{k,n}]] 
Wu3[G]==Total[Flatten[Table[w[G[[k]]]*w[G[[l]]]*Wu3[B2[[k,l]]],{k,n},{l,n}]]]

\end{lstlisting}
\end{tiny}

\bibliographystyle{plain}

\begin{thebibliography}{10}

\bibitem{Alexandroff1937}
P.~Alexandroff.
\newblock Diskrete {R\"aume}.
\newblock {\em Mat. Sb. 2}, 2, 1937.

\bibitem{DehnHeegaard}
M.~Dehn and P.~Heegaard.
\newblock Analysis situs.
\newblock {\em Enzyklopaedie d. Math. Wiss}, III.1.1:153--220, 1907.

\bibitem{KnillKuenneth}
O.~Knill.
\newblock The {K\"u}nneth formula for graphs.
\newblock {{\\}http://arxiv.org/abs/1505.07518}, 2015.

\bibitem{DehnSommerville}
O.~Knill.
\newblock On a {D}ehn-{S}ommerville functional for simplicial complexes.
\newblock {\\}https://arxiv.org/abs/1705.10439, 2017.

\bibitem{CountingMatrix}
O.~Knill.
\newblock The counting matrix of a simplicial complex.
\newblock {\\}https://arxiv.org/abs/1907.09092, 2019.

\bibitem{EnergizedSimplicialComplexes}
O.~Knill.
\newblock Energized simplicial complexes.
\newblock https://arxiv.org/abs/1908.06563, 2019.

\bibitem{KnillEnergy2020}
O.~Knill.
\newblock The energy of a simplicial complex.
\newblock {\em Linear Algebra and its Applications}, 600:96--129, 2020.

\bibitem{GreenFunctionsEnergized}
O.~Knill.
\newblock Green functions of energized complexes.
\newblock {\\}https://arxiv.org/abs/2010.09152, 2020.

\bibitem{FiniteTopology}
O.~Knill.
\newblock Finite topologies for finite geometries.
\newblock https://arxiv.org/abs/2301.03156, 2023.

\bibitem{Sphereformula}
O.~Knill.
\newblock The sphere formula.
\newblock https://arxiv.org/abs/2301.05736, 2023.

\end{thebibliography}

\end{document}